 \documentclass[a4paper,12pt]{amsart}
 \usepackage{amsmath,amsthm,amssymb,color}
 \usepackage{graphicx}
 \usepackage[all]{xy} 
 \usepackage{pxfonts}
 \usepackage{color}
 \usepackage[margin=2cm]{geometry} 
 \input xypic
\usepackage{amssymb} 

\theoremstyle{definition}
\newtheorem{Thm}{Theorem}[section]
\newtheorem{Def}[Thm]{Definition}

\newtheorem{Rem}[Thm]{Remark}
\newtheorem{Ex}[Thm]{Example}
\newtheorem{Lem}[Thm]{Lemma} 
\newtheorem{Cor}[Thm]{Corollary} 
\newtheorem{Prop}[Thm]{Proposition}

\newcommand{\Q}{\mathbb{Q}}

\newcommand{\Z}{\mathbb{Z}}

\newcommand{\F}{\mathbb{F}}

\newcommand{\Tel}{\mathrm{Tel}}

\newcommand{\rank}{\mathrm{rank}}
\newcommand{\Sq}{\mathrm{Sq}} 
\newcommand{\cptwo}{\ensuremath{\mathbb{C}P^{2}}}

\newcommand{\cohlgy}[1]{\ensuremath{H^{*}(#1)}} 

\def\red#1{{\textcolor{red}{#1}}}

\newcounter{bean}

\newcommand{\namedright}[3]{\ensuremath{#1\stackrel{#2}
 {\longrightarrow}#3}}
\newcommand{\nameddright}[5]{\ensuremath{#1\stackrel{#2}
 {\longrightarrow}#3\stackrel{#4}{\longrightarrow}#5}}
\newcommand{\namedddright}[7]{\ensuremath{#1\stackrel{#2}
 {\longrightarrow}#3\stackrel{#4}{\longrightarrow}#5
  \stackrel{#6}{\longrightarrow}#7}}

\newcommand{\qqed}{\hfill\Box}

\begin{document} 

\title{Suspension splittings and self-maps of flag manifolds}
\author{Shizuo Kaji} 
\thanks{This work was supported by KAKENHI, Grant-in-Aid for Scientific Research (C) 18K03304.}
\address{Institute of Mathematics for Industry, Kyushu University, Fukuoka, 819-0395, Japan.}
\email{skaji@imi.kyushu-u.ac.jp}
\author{Stephen Theriault}
\address{Mathematical Sciences, University of Southampton, 
     Southampton SO17 1BJ, United Kingdom}
\email{s.d.theriault@soton.ac.uk}

\subjclass[2010]{Primary 55P40, 55S37, Secondary 57T15}
\keywords{flag manifold, self-map, stable splitting} 
\date{\today}


\begin{abstract}
If $G$ is a compact connected Lie group and $T$ is a maximal torus, we 
give a wedge decomposition of $\Sigma G/T$ by identifying families of 
idempotents in cohomology. This is used to give new information on the 
self-maps of $G/T$. 
\end{abstract}

\maketitle

\section{Introduction} 
Let $G$ be a compact connected Lie group and let $T$ be a 
maximal torus. There has been considerable interest in trying to determine 
the homotopy classes of the self-maps of the quotient space $G/T$. 
One method commonly adopted in~\cite{DZ,GH,P,Z1} is to study the image of the map 
\[r\colon\namedright{[G/T,G/T]}{}{\mathrm{Hom}_{\mathrm{alg}}(\cohlgy{G/T},\cohlgy{G/T})}.\] 
We show that if $G$ is simply-connected 
 there is a bijection 
\begin{equation} 
  \label{selfmapsplit} 
  [G/T,G/T]\cong [G/T,G]\times\, \mbox{Im}(r), 
\end{equation} 
where $r$ sends a self-map to the ring homomorphism it induces in cohomology. 
This was earlier claimed to hold in more generality in~\cite{Z2}, but there 
seems to be gaps. With~(\ref{selfmapsplit}) in hand, we consider the other factor, 
$[G/T,G]$, and develop an approach to understanding it. 

Since $G$ is a group it has a classifying space $BG$. This implies that there is 
a group isomorphism $[G/T,G]\cong [\Sigma G/T,BG]$. The idea is to decompose 
$\Sigma G/T$ into a wedge of smaller spaces which simplify the calculations. 
The decompositions are obtained by identifying certain idempotents 
in cohomology. These are $p$-local decompositions, where $p$ is a prime. 
Two families of idempotents are considered, one coming from 
Adams operations on the classifying spaces of $T$ and $G$, the other 
coming from the action of the Weyl group on~$G/T$. 
These are also compatible in the sense that they can be merged to form 
a larger set of idempotents, giving a finer decomposition of the space. 

These suspension splittings also fit into a larger framework that considers stable 
decompositions of homogeneous spaces. The classic example of this is 
Miller's stable splitting of Stiefel manifolds~\cite{M}, which inspired a great 
many variants and refinements (e.g., \cite{K,NY,Ullman,Y}). In those cases, the stable feature is prominent 
in the sense that multiple suspensions are usually needed to realize the 
decomposition, whereas in our case the decomposition occurs after a single 
suspension. 

To demonstrate the methods we give explicit decompositions of 
$SU(3)/T$, $SU(4)/T$ , $Sp(2)/T$ and $G_{2}/T$, and go on to calculate 
$[G/T,G]$ in each case (modulo $2$-primary information in the $SU(4)$ 
and $G_{2}$ cases). 

\section{The cohomology of $G/T$}
Let $W=N(T)/T$ be the Weyl group of $G$, which is generated by the
simple reflections $s_1,\ldots,s_r$, where $r=\rank(T)$.
\begin{Def}
For an element of $w\in W$, the length $l(w)$ of $w$ is the least integer such 
that~$w$ can be written as a product of $l(w)$ simple reflections. So an element $w$ of 
length $n$ can be written as $w=s_{i_1} s_{i_2} \cdots s_{i_n}$ for some sequence 
of simple reflections. We often abbreviate this as $w=s_{i_1i_2\cdots i_n}$.
\end{Def} 

The following theorem proved by~\cite{Che} describes $H^{\ast}(G/T;\Z)$ as a free $\Z$-module. 

\begin{Thm}[Bruhat decomposition] 
\label{Bruhat} 
There is a cell decomposition
\[
 G/T = \bigcup_{w\in W} \sigma_w,
\]
where $\sigma_w$ are open cells with $\dim(\sigma_w)=2l(w)$. 
Moreover, the closure is $\overline{\sigma_w}=\bigcup_{v\le w} \sigma_v$,
where the order on $W$ is given by the strong Bruhat order,
that is, $v\le w$ iff a reduced word for $w$ contains one of $v$ as a sub-word.
Consequently, $H^*(G/T;\Z)$ is torsion free of rank $|W|$ and its basis is given 
by the Schubert classes:
\[
 H^*(G/T;\Z) \cong H^{even}(G/T;\Z) \simeq \Z \langle \sigma_w \rangle_{w\in W}.
\]
\end{Thm} 
\vspace{-0.7cm}~$\qqed$\bigskip 

Borel~\cite{Borel1953} gives another description of $H^{\ast}(G/T)$
as the quotient of a polynomial ring.
\begin{Thm}[Coinvariant description]\label{thm:coinvariant}
Let $R$ be a ring in which the torsion primes \cite{Borel1961} of $G$ are inverted.
Then
\[
H^\ast(G/T;R) = H^\ast(BT;R)/I 
\] 
where $I$ is the ideal generated by the Weyl group invariants of positive degree.
\end{Thm} 

\begin{Ex}
\[
H^*(SU(n)/T^{n-1};\Z) = \Z[x_1,x_2,\ldots, x_n]/(e_1,e_2,\ldots,e_n),
\]
where $e_i$ is the $i$-th elementary symmetric function on $x_1,\ldots,x_n$.
A choice of a polynomial representative of $\sigma_w$ in this presentation
is given by the classical {\em Schubert polynomial} \cite{LS}.
\end{Ex}

\section{A splitting of $[G/T,G/T]$} 
The group homomorphism 
\(\namedright{T}{}{G}\) 
classifies, giving a homotopy fibration sequence 
\[\namedddright{T}{}{G}{q}{G/T}{j}{BT}\longrightarrow BG\] 
which defines the map $j$. In particular, for a simply-connected space $X$ 
we obtain an exact sequence 
\[\nameddright{[X,G]}{q_*}{[X,G/T]}{j_{\ast}}{[X.BT]},\]
where $q_*$ is injective. 
Consider the map 
\[r\colon\namedright{[X,G/T]}{}{\mathrm{Hom}_{\mathrm{alg}}(H^*(G/T;\Z),H^*(X;\Z))}\]  
defined by sending a map 
\(\namedright{X}{}{G/T}\) 
to the algebra homomorphism it induces in cohomology. Similarly, there is a map 
\[\namedright{[X,BT]}{}{\mathrm{Hom}_{\mathrm{alg}}(H^*(BT;\Z),H^*(X;\Z))}\] 
which is an isomorphism since both sides are canonically isomorphic to 
$\bigoplus_{1\le i\le r} H^2(X;\Z)$. We obtain a commutative diagram
\begin{equation} 
\label{Jdgrm} 
\xymatrix{
[X,G] \ar@{^{(}-_>}[r] & [X,G/T] \ar[r]^{j_*} \ar[d]^r & [X,BT] \ar[d]^\simeq \\
& \mathrm{Hom}_{\mathrm{alg}}(H^*(G/T;\Z),H^*(X;\Z)) \ar[r]^J & \mathrm{Hom}_{\mathrm{alg}}(H^*(BT;\Z),H^*(X;\Z)) 
}
\end{equation}  
where $J(f^{\ast})=f^{\ast}\circ j^{\ast}$. 
We have the holonomy action $[X,G/T]\times [X,G]\to [X,G/T]$.
By the Puppe sequence~\cite[Lemma 1.4.7]{maypont},
 the action is free and moreover,
for $f_1, f_2 \in [X,G/T]$
 we have $j_*(f_1)=j_*(f_2)$ if and only if $f_1$ and $f_2$ are in the same orbit of the action of $[X,G]$. 
Therefore, we have the following non-canonical identification
\begin{equation} 
\label{noncanonsplitting} 
[X,G/T] \cong [X,G]\times \mathrm{Im}(j_*).
\end{equation} 

\begin{Prop} 
   \label{selfsplit} 
   If $X$ is simply-connected and $H^*(X;\Z)$ is torsion-free then the map $J$ is 
   a monomorphism. 
\end{Prop} 

\begin{proof} 
Since $H^*(G/T;\Z),H^*(X;\Z)$, and $H^*(BT;\Z)$ are torsion free,
the vertical maps (rationalizations) in the following commutative diagram are injective
\[
\xymatrix{
\mathrm{Hom}_{\mathrm{alg}}(H^*(G/T;\Z),H^*(X;\Z)) \ar[r]^J \ar[d] &\mathrm{Hom}_{\mathrm{alg}}(H^*(BT;\Z),H^*(X;\Z)) \ar[d]\\
\mathrm{Hom}_{\mathrm{alg}}(H^*(G/T;\Q),H^*(X;\Q)) \ar[r]^{J_{(0)}} &\mathrm{Hom}_{\mathrm{alg}}(H^*(BT;\Q),H^*(X;\Q)), 
}
\]
where $J_{(0)}$ is the rationalization of $J$. Since $H^*(G/T;\Q)$ is generated by the degree 
two elements and $H^2(G/T;\Q)\cong H^2(BT;\Q)$, we see that $J_{(0)}$ is injective. The 
commutativity of the diagram then implies that $J$ is also injective.
\end{proof} 

\begin{Rem} 
Zhao~\cite[Lemma 1]{Z2} claims that $J$ is injective without the torsion-free hypothesis.
However, this seems unlikely. Observe that $H^*(G/T;\Z)$ is torsion-free and $H^*(G/T;\Q)$ is 
generated by degree two elements. But $H^*(G/T;\Z)$ is NOT generated by degree two elements 
in general; for example in the case of $H^*(G_2/T;\Z)$ described in \S \ref{G2section},
even when an induced map~$f^*$ 
for $f\in [X,G/T]$ is trivial on $H^2(G_2/T;\Z)\simeq \Z[x_1,x_2,x_3]/(e_1)$, $f^*(\gamma)$ can be a non-trivial 
torsion element in $H^*(X;\Z)$. 
\end{Rem} 

\begin{Cor} 
   \label{selfsplitcor} 
   If $X$ is simply-connected and $H^*(X)$ is torsion-free then there is an isomorphism 
   \[[X,G/T]\cong [X,G]\times\mbox{Im}(r).\] 
\end{Cor} 

\begin{proof} 
By Proposition~\ref{selfsplit}, the map $J$ in~(\ref{Jdgrm}) is a monomorphism. 
The square in~(\ref{Jdgrm}) therefore implies that $\mbox{Im}(j_{\ast})\cong\mbox{Im}(r)$. 
Now substitute this isomorphism into~(\ref{noncanonsplitting}). 
\end{proof} 
   
By Theorem~\ref{Bruhat}, $H^*(G/T)$ is torsion-free. So Proposition~\ref{selfsplit} 
immediately implies the following. 

\begin{Cor} 
   \label{selfGTsplit} 
   Let $G$ be a compact simply-connected Lie group. Then there is an isomorphism 
   \[[G/T,G/T]\cong [G/T,G]\times\mbox{Im}(r).\]  
\end{Cor} 
\vspace{-0.7cm}~$\qqed$\bigskip 
   
Note that $G/T\simeq \hat{G}/\hat{T}$,
where $\hat{G}$ is the universal cover of $G$ and 
$\hat{T}$ is the maximal torus of $\hat{G}$.
Hence, we can always take $G$ to be simply-connected.

\section{Idempotents for $H^{\ast}(G/T)$} 

Now we start to focus on $[G/T,G]$, which by Corollary~\ref{selfGTsplit} is a factor 
of $[G/T,G/T]$. In this section we construct two families of compatible idempotents for 
$H^{\ast}(G/T)$ and use them to produce wedge decompositions of $\Sigma G/T$. 
This begins with a general lemma~(c.f. \cite[\S 2]{Priddy}).

\begin{Def}
Let $R$ be a ring. 
A finite collection $p_1,p_2,\ldots,p_n$ of self-maps of a connected space $X$ 
is called a set of \emph{mutually orthogonal idempotents} of $H^{\ast}(X;R)$ if: 
\begin{itemize} 
   \item[(i)] $p_{i}^{\ast}\circ p_{i}^{\ast}=p_{i}^{\ast}$ for $1\leq i\leq n$;\smallskip  
   \item[(ii)] $p_{i}^{\ast}\circ p_{j}^{\ast}=0$ for all $1\leq i,j\leq n$ with $i\neq j$; and\smallskip  
   \item[(iii)] $p_{1}^{\ast}+\cdots +p_{n}^{\ast}=1$. 
\end{itemize} 
\end{Def} 

Given a self-map 
\(f\colon\namedright{X}{}{X}\), 
let $\Tel(f)$ be the telescope of $f$ and let 
\(t\colon\namedright{X}{}{\Tel(f)}\) 
be the map to the telescope. Since $t\circ f\simeq t$, the map 
\(\namedright{H^{\ast}(\Tel(f);R)}{t^{\ast}}{H^{\ast}(X;R)}\) 
induces the inclusion of \mbox{Im}\,$(f^{\ast})$. 

\begin{Lem} 
\label{idemdecomp} 
Let $X$ be a simply-connected finite co-$H$-space.
Let $p_{1},\ldots,p_{n}$ be a set of mutually orthogonal idempotents on 
$H^{\ast}(X;\Z/p\Z)$. Then there is a $p$-local homotopy equivalence 
\[
X \simeq \bigvee_{i=1}^{n}\Tel(p_i).
\]
\end{Lem} 

\begin{proof} 
Since 
\(\namedright{X}{p_{i}}{\Tel(p_{i})}\) 
induces the inclusion of \mbox{Im}\,$(p_{i}^{\ast})$, 
the sum of the maps $p_{i}$ defines a map 
\(\psi\colon\namedright{X}{}{\bigvee_{i=1}^{n}\Tel(p_{i})}\) 
which induces an isomorphism in mod-$p$ cohomology. Since $X$ 
is simply-connected and of finite type, this implies that $\psi$ is a $p$-local homotopy equivalence
by \cite[Chapter II, Theorem 1.14]{HMR}.
\end{proof}

We identify two families of self maps of $\Sigma G/T$ that can be used to produce 
idempotents on $H^\ast(\Sigma G/T;\Z/p\Z)$.
Note that the co-$H$ structure on $\Sigma G/T$ induces a group structure on
$[\Sigma G/T, \Sigma G/T]$.

\subsection{Unstable Adams operations}\label{sec:adams}
We follow the argument in~\cite{Y}. For $l\in \Z$ prime to $|W|$, there is a commutative diagram 
\[
\xymatrix{
G \ar[r] \ar[d]^{\Omega \psi^l} & G/T \ar[r] \ar[d]^{\psi^l} & BT \ar[r] \ar[d]^{\psi^l} & BG \ar[d]^{\psi^l}\\
G \ar[r] & G/T \ar[r] & BT \ar[r] & BG
}
\]
where $(\psi^l)^*: H^{2i}(X) \to H^{2i}(X)$ is multiplication by $l^i$.

For an odd prime $p$, choose $l\in \Z$ which is primitive in $\F_p^\times$ and define 
a self-map of $\Sigma G/T$ by
\[
\phi_i=\Sigma \psi^l-l^i \text{ and } 
\varphi'_k = \prod_{1\le i\le n,\ i\not\equiv k \bmod{p-1}} \phi_i,
\]
where $n=\max(p-1,\dim(G/T)/2)$ and $l^i: \Sigma G/T \to \Sigma G/T$ is $l^i$ times the identity map.
Note that $\varphi'_i$ is trivial on $H^{2j+1}(\Sigma G/T;\Z/p\Z)$ iff $j=i \bmod{p-1}$.
So by normalizing up to unit, we obtain a set of mutually orthogonal idempotents 
$\varphi_{1},\ldots,\varphi_{n}$ on $H^{\ast}(\Sigma G/T;\Z/p\Z)$, where $\varphi_{i}=u_{i}\varphi'_{i}$ 
for some unit $u_{i}\in\Z/p\Z$. Therefore, by Lemma~\ref{idemdecomp} there is a 
$p$-local homotopy equivalence  
\[
 \Sigma G/T \simeq \bigvee \Tel(\varphi_i),
\]
where 
$\tilde{H}^{2i+1}(\Tel(\varphi_k);\Z/p\Z) \cong 
\begin{cases} \tilde{H}^{2i}(G/T;\Z/p\Z) & \mbox{if $i=k \mod p-1$} \\
0 & \mbox{if $i\neq k \mod p-1$}. \end{cases}$


\subsection{The Weyl group action}\label{sec:weyl}
The flag manifold $G/T$ is equipped with a right Weyl group action:
\[
 gT \mapsto gwT
\]
for $w\in W=N(T)/T$. 
Thus, given any $w\in W$ we obtain a self-map 
\(w\colon\namedright{G/T}{}{G/T}\). 
In particular, each simple reflection $s_{i}$ induces a self-map 
\(s_{i}\colon\namedright{G/T}{}{G/T}\). 

By~\cite{BGG}, the $W$-action on Schubert classes is given by 
\[
 s_i \sigma_w = \begin{cases}
 \sigma_w & \mbox{if $l(ws_i) = l(w)+1$} \\
 -\sigma_w - \sum_{l(w s_i s_\beta)=l(w) } \dfrac{2(\beta,\alpha_i)}{(\beta,\beta)} \sigma_{w s_\beta} & \mbox{if $l(ws_i) = l(w) -1$}.
 \end{cases}
\]
In the coinvariant description (Theorem \ref{thm:coinvariant}),
the $W$-action is simply induced by the ordinary one on $H^*(BT;R)$. 

Using the co-$H$-structure on $\Sigma G/T$ to add maps,
to each element $v$ in the group ring $\Z[W]$ there associated a self-map  
\(v\colon\namedright{\Sigma G/T}{}{\Sigma G/T}\). 
Thus if we find a set of mutually orthogonal idempotents in the group ring 
we can find an induced set of mutually orthogonal idempotents in $H^{\ast}(\Sigma G/T;\mathbb{Z})$. 
The same argument works if we replace $\mathbb{Z}$-coefficients with $\Z/p\Z$ 
or $\mathbb{Q}$-coefficients and consider the corresponding localization of the space. 

It is well-known that $H^{\ast}(G/T;\Q)$ is the regular representation of $W$ and
decomposes into irreducible representations.
However, constructing the corresponding set of mutually orthogonal idempotents even 
in $\Q[W]$ is non-trivial \cite{bergeron}. For our purpose, we aim to construct 
mutually orthogonal idempotents in $\Z/p\Z[W]$, and in the examples in 
Section~\ref{sec:examples}, the identification of the mutually orthogonal idempotents is ad hoc. 

Nevertheless, given a set of mutually orthogonal idempotents $\{c_{1},\ldots,c_{n}\}$ on 
$H^{\ast}(\Sigma G/T;\Z/p\Z)$, by Lemma~\ref{idemdecomp} we obtain 
a $p$-local homotopy equivalence 
\[\Sigma G/T\simeq\bigvee_{i=1}^{n}\Tel(c_{i}).
\]  

\subsection{Putting the two decompositions together} 
In general, if $p_{1},\ldots,p_{n}$ and $q_{1},\ldots,q_{m}$ are two sets of 
mutually orthogonal idempotents on $H^{\ast}(X;\Z/p\Z)$ that commute, 
where $X$ is a simply-connected finite co-$H$-space, 
then the collection $\{p_{i}\circ q_{j}\mid 1\leq i\leq n, 1\leq j\leq m\}$ is another 
set of mutually orthogonal idempotents on $H^{\ast}(X;\Z/p\Z)$. In our case, 
the idempotents $\varphi_{i}^{\ast}$ from the unstable Adams operations and the 
idempotents~$c_{j}^{\ast}$ from the action of the Weyl group commute 
since $\varphi_i^{\ast}$ is just a projection on to the subspaces consisting of elements 
of specific degrees while $c_j^{\ast}$ preserves the degrees. Thus the maps 
$\{\varphi_{i}\circ c_{j}\mid 1\leq i\leq n, 1\leq j\leq m\}$ form a set of mutually 
orthogonal idempotents on $H^{\ast}(\Sigma G/T;\Z/p\Z)$, and produce a finer 
decomposition of $\Sigma G/T$. We think of the decomposition based on the unstable
Adams operation as splitting $H^{\ast}(\Sigma G/T;\Z/p\Z)$ ``horizontally'' while
the one based on Weyl group action splits ``vertically.'' 

Note that as any space rationally splits into a wedge of spheres after suspension, 
we are primarily interested in $p$-local decompositions of $\Sigma G/T$ for a small prime $p$.

\section{Examples} 
\label{sec:examples} 

In identifying homotopy 
types of telescopes, we freely use the fact that the element $\eta\in\pi_{n+1}(S^{n})$ 
is detected by the Steenrod operation $Sq^{2}$ and at odd primes the element 
$\alpha_{1}\in\pi_{n+2p-3}(S^{n})$ is detected by the Steenrod operation $\mathcal{P}^{1}$. 
This is equivalent to saying that if $\tilde{H}^{\ast}(\Tel(c_{i});\Z/2\Z)\cong\Z/2\Z\{x,Sq^{2}(x)\}$ 
for $\vert x\vert=d\geq 3$ then there is a $2$-local homotopy equivalence 
$\Tel(c_{i})\simeq\Sigma^{d-2}\cptwo$, and if $p$ is odd and 
$\tilde{H}^{\ast}(\Tel(c_{i});\Z/p\Z)\cong\Z/p\Z\{x,\mathcal{P}^{1}(x)\}$ for $\vert x\vert=d\geq 3$ 
then there is a $p$-local homotopy equivalence 
$\Tel(c_{i})\simeq A(d,d+2p-2)$ where $A(d,d+2p-2)$ is the homotopy cofiber of 
\(\namedright{S^{d+2p-3}}{\alpha_{1}}{S^{d}}\). 

In what follows, for a fixed prime $p$, we generically use the notation 
$c_{1},\ldots,c_{n}$ for a set of mutually orthogonal idempotents in the group ring $\Z/p\Z[W]$
 and let $V_1,\ldots,V_n$ be their images in $H^\ast(G/T;\Z/p\Z)$.  
By abuse of notation, we use the same symbol $c_i$ to denote the corresponding 
idempotents on $H^{\ast}(\Sigma G/T;\Z/p\Z)$. 
The action of the Steenrod operations $Sq^{2}$ or $\mathcal{P}^{1}$ 
are determined by, for example, \cite{DZ2007}. 
Computation on cohomology is carried out with the aid of a computer code described in \cite{maple}.
Idempotents in the group ring are obtained by solving quadratic equations in prime fields. 

\subsection{Type $A_n$ case}
For the type $A_n$-case, non-modular irreducible representations $V_\lambda$ are 
obtained by considering the Young symmetrizers for all the standard tableaux of shape $\lambda$. 
However, they are not always mutually orthogonal \cite{Stembridge2010}. 
We will look at some low rank cases in an ad hoc way. 

\begin{Ex}\label{SU(3)-weyl}
Since $SU(2)/T=S^2$, the simplest non-trivial case is when $G=SU(3)$.
By Theorems \ref{Bruhat} and \ref{thm:coinvariant},
 $H^*(SU(3)/T^2;\Z) \simeq \Z[x_1,x_2,x_3]/(e_1,e_2,e_3) \simeq \Z\langle 1,\sigma_1,\sigma_2,\sigma_{12},\sigma_{21},\sigma_{121}\rangle$.
It is well-known that the type-$A_r$ Weyl group is the symmetric group $S_{r+1}$.
In particular, $W= \langle s_1, s_2 \rangle$,  where the simple reflection $s_i$ swaps $x_i$ and $x_{i+1}$.
The Schubert cell decomposition looks like
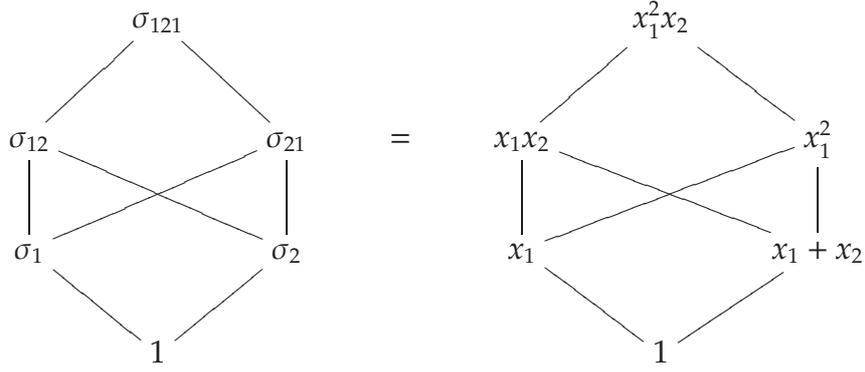
\[
 \xymatrix{
 & \sigma_{121} \ar@{-}[rd] \ar@{-}[ld] & 
     & & & x_1^2 x_2 \ar@{-}[rd] \ar@{-}[ld] & \\
 \sigma_{12} \ar@{-}[d] \ar@{-}[drr]& & \sigma_{21} \ar@{-}[d] \ar@{-}[dll] 
     & = & x_1x_2 \ar@{-}[d] \ar@{-}[drr]& & x_1^2 \ar@{-}[d] \ar@{-}[dll]\\
 \sigma_{1} \ar@{-}[rd] & & \sigma_{2} \ar@{-}[ld] 
   & & x_1 \ar@{-}[rd] & & x_1+x_2 \ar@{-}[ld] \\
 & 1 &  & & & 1 & 
}
\]
where $\sigma_1^3=\sigma_2^3=0$.
The horizontal level indicates the degree of the cells,
and the lines indicate possible non-trivial attaching maps.

We first find mutually orthogonal idempotents in the group ring (see \S \ref{sec:weyl}).
For $p=2$, there is a set of mutually orthogonal idempotents in $\Z/2\Z[W]$ 
\begin{align*}
c_1 & = 1+s_{12}+s_{21} \\  
c_2 & = 1+s_2+s_{21}+s_{121} \\
c_3 & =  1+s_2+s_{12}+s_{121} 
\end{align*}
satisfying 
\begin{align*} 
V_1 & =\langle 1, \sigma_{121} \rangle \\ 
V_2 & =\langle \sigma_{1}+\sigma_{2},\sigma_{12}+\sigma_{21} \rangle \\ 
V_3 & =\langle \sigma_{1},\sigma_{21} \rangle 
\end{align*} 
where $\Sq^2(\sigma_1)=\sigma_{21}, \Sq^2(\sigma_2)=\sigma_{12}$. Therefore 
there is a $2$-local homotopy equivalence 
\[\Sigma SU(3)/T^{2}\simeq_2 S^{7}\vee\Sigma\cptwo\vee\Sigma\cptwo.\] 

For $p=3$, there is a set of mutually orthogonal idempotents in $\Z/3\Z[W]$
\begin{align*}
c_1 &= 2+s_{121} \\
c_2 &= 1-c_1  
\end{align*} 
satisfying 
\begin{align*} 
V_1 &= \langle 1, \sigma_1+2\sigma_2, \sigma_{12}+\sigma_{21} \rangle \\
V_2 &= \langle \sigma_1+\sigma_2, \sigma_{12}+2\sigma_{21},\sigma_{121} \rangle 
\end{align*}
where there is no non-trivial $\mathcal{P}^1$. Therefore each telescope is $3$-locally 
homotopy equivalent to a wedge of spheres and we obtain a $3$-local homotopy equivalence 
\[\Sigma SU(3)/T^{2}\simeq_3 S^{3}\vee S^{3}\vee S^{5}\vee S^{5}\vee S^{7}.\] 

If $p>3$, for degree reasons the unstable Adams operations (\S \ref{sec:adams}) imply that there is a $p$-local 
homotopy equivalence 
\[
\Sigma SU(3)/T^{2} \simeq_p S^3 \vee S^3 \vee S^5 \vee S^5 \vee S^7.
\]
\end{Ex}

The next example is $SU(4)/T^{3}$. It will be helpful to have some splitting 
information that comes from geometry as well as Adams operations and 
the action of the Weyl group. 

\begin{Lem}\label{lem:top-split}
The stable normal bundle of any flag manifold $G/T$ is trivial.
In particular, the top cell of $G/T$ stably splits off.
\end{Lem}
\begin{proof}
Let $\mathfrak{g}$ and $\mathfrak{t}$ be the Lie algebras of $G$ and $T$ respectively.
Take a regular element $X\in \mathfrak{g}$ and consider the adjoint embedding
$G/T\to \mathfrak{g}$ induced by $g \mapsto Ad_g(X)$.
The normal bundle $\nu$ is $G\times_T \mathfrak{t}$, which is trivial since the adjoint action of $T$ on $\mathfrak{t}$ is trivial.
By the Pontrjagin-Thom construction combined with the pinching map, we obtain the splitting
\[
S^{\dim(G)}\to (G/T)^\nu \simeq \Sigma^{\dim(T)} (G/T)_+ \to S^{\dim(G)},
\]
where $(G/T)^\nu$ is the Thom complex of the normal bundle $\nu$.
\end{proof}

\begin{Ex}\label{SU(4)-weyl}
The rank of $H^*(SU(4)/T^3;\Z)\simeq \Z[x_1,x_2,x_3,x_4]/(e_1,e_2,e_3,e_4)$ is
$\{1,3,5,6,5,3,1\}$.
The Weyl group $W$ is the symmetric group $\langle s_1, s_2, s_3 \rangle$ with $|W|=24$. 

For $p=2$, there do not appear to be so many idempotents. For example, $c_1=s_{23}+s_{32}$ 
and $c_2=1-c_1$ form a mutually orthogonal set of idempotents in $\Z/2\Z[W]$
but they do not help much in terms of producing splittings with identifiable wedge summands. 

For $p=3$, there is a set of mutually orthogonal idempotents in $\Z/3\Z[W]$
\begin{align*}
c_1 &= (1+s_1+s_2+s_{12}+s_{21}+s_{121})(1-s_{12321}) \\
c_2 &= (1+s_{12321})(1-s_1-s_2+s_{12}+s_{21}-s_{121}) \\
c_3 &= 2+2s_{1}+2s_{2}+2s_{3}+2s_{2 1}+2s_{1 2}+2s_{3 2}+2s_{2 3}+2s_{1 2 1}+s_{3 2 1}+2s_{2 3 2}+s_{1 2 3}+s_{1 3 2 1}+s_{2 3 2 1}+s_{1 2 1 3} \\ 
       & \hspace{1cm} +s_{1 2 3 2}+s_{2 1 3 2 1}+s_{1 2 3 2 1}+s_{1 2 1 3 2}+s_{1 2 1 3 2 1} \\
c_4 &= 2+s_2+2s_{13}+s_{213}+s_{132}+2s_{2132}+s_{12321}+2s_{121321}\\
c_5 &= 2+2s_2+2s_{13}+2s_{213}+2s_{132}+2s_{2132}+2_{12321}+2s_{121321}\\
c_6 &= 2+s_{1}+s_{2}+2s_{3}+2s_{2 1}+s_{1 3}+2s_{1 2}+s_{3 2}+s_{2 3}+s_{1 2 1}+s_{3 2 1}+2s_{2 1 3}+2s_{1 3 2}+2s_{2 3 2}+2s_{1 2 3}+2s_{1 3 2 1} \\ 
       & \hspace{1cm} +2s_{2 3 2 1}+s_{1 2 1 3}+s_{2 1 3 2}+s_{1 2 3 2}+s_{2 1 3 2 1}+s_{1 2 3 2 1}+2s_{1 2 1 3 2}+2s_{1 2 1 3 2 1}\\
c_7 &= 2+s_{2 3}+s_{3 2}+s_{3}+s_{1 3}+s_{2 1 3}+s_{1 3 2}+s_{2 3 2}+s_{2 1 3 2}+2s_{1 2 1 3}+2s_{1 2 1 3 2}+2s_{1 2 3}+2s_{1 2 3 2}+s_{1 3 2 1}+s_{3 2 1} \\ 
       & \hspace{1cm} +s_{2 1 3 2 1}+2s_{1 2 1 3 2 1}+s_{2 3 2 1}+2s_{1 2 3 2 1}+2s_{1 2 1}+2s_{2 1} 
  +2s_{1 2}+2s_{2}+2s_{1}\\
c_8 &= 1-\sum_{i=1}^{7} c_i
\end{align*}
satisfying 
\begin{align*}
V_1 &= \langle \sigma_3, \sigma_{23}, \sigma_{123}\rangle \\
V_2 &= \langle \sigma_{123}+\sigma_{121}+\sigma_{321}+2\sigma_{213}+2\sigma_{132}+\sigma_{232},
\sigma_{1213}+\sigma_{1232}+2\sigma_{1321}+2\sigma_{2321},
\sigma_{21321}+\sigma_{12321}+\sigma_{12132} \rangle \\
V_3 &= \langle \sigma_{321}+2\sigma_{213}+\sigma_{123},
\sigma_{2321}+2\sigma_{1213},\sigma_{12321}
\rangle \\
V_4 &= \langle 2\sigma_{12}+\sigma_{13}+\sigma_{21}+2\sigma_{32}+\sigma_{23},
2\sigma_{2321}+\sigma_{2132}+2\sigma_{1213},\sigma_{121321} \rangle \\
V_5 &= \langle 1,\sigma_{12}+\sigma_{13}+\sigma_{23},
\sigma_{2321}+\sigma_{1213} \rangle \\
V_6 &= \langle 
\sigma_{321}+2\sigma_{213}+2\sigma_{132}+2\sigma_{232}+\sigma_{123},
\sigma_{2321}+\sigma_{1232}+2\sigma_{2132}+2\sigma_{1213},
\sigma_{12321}+2\sigma_{12132} \rangle\\
V_7 &= \langle \sigma_{1}+2\sigma_{3}, \sigma_{21}+\sigma_{23}, \sigma_{321}+2\sigma_{123}\rangle \\
V_8 &= \langle \sigma_{1}+\sigma_{2}+\sigma_{3},\sigma_{21}+2\sigma_{12}+\sigma_{32}+2\sigma_{23},
\sigma_{321}+2\sigma_{132}+\sigma_{213} \rangle.
\end{align*}
The non-trivial actions of $\mathcal{P}^1$ are
\begin{align*}
V_1, V_7, V_8: & H^2\to H^6 \\
V_4, V_5: & H^4 \to H^8 \\
V_2,V_3,V_6: &  H^6\to H^{10}.
\end{align*} 
For degree reasons in cohomology, the unstable Adams operations split $\Tel(c_{i})$ 
for $i\in\{1,7,8\}$ into wedge summands: one inheriting the degree $3$ and $7$ generators 
in cohomology and the other inheriting the degree~$5$ generator. Similarly, the unstable 
Adams operations split $\Tel(c_{i})$ for $i\in\{2,3,6\}$ into wedge summands: one 
inheriting the degree $7$ and $11$ generators in cohomology and the other inheriting the 
degree~$9$ generator. All together we obtain $3$-local homotopy equivalences 
\begin{align*} 
     \Tel(c_{1})\simeq\Tel(c_{7})\simeq\Tel(c_{8}) &\simeq A(3,7)\vee S^{5} \\ 
     \Tel(c_{2})\simeq\Tel(c_{3})\simeq\Tel(c_{6}) & \simeq A(7,11)\vee S^{9}. 
\end{align*} 
Note that $\Tel(c_{5})\simeq A(5,9)$ and $\Tel(c_{4})$ is a three-cell complex 
whose $9$-skeleton is $A(5,9)$. We claim that $\Tel(c_{4})\simeq A(5,9)\vee S^{13}$. 
To see this we show that the attaching map 
\(g\colon\namedright{S^{12}}{}{A(5,9)}\) 
for the top cell is null homotopic. Since there is no Steenrod operation connecting 
the $9$ and $13$-cells in $\Tel(c_{4})$, the composition 
\(\nameddright{S^{12}}{g}{A(5,9)}{q}{S^{9}}\) 
must be null homotopic, where $q$ is the pinch map to the top cell. 
Therefore $g$ lifts to the homotopy fibre of $q$, which by a Serre spectral 
sequence argument, is homotopy equivalent to $S^{5}$ in dimensions~$\leq 12$. 
Thus~$g$ is homotopic to a composite 
\(\namedright{S^{12}}{t\cdot\alpha_{2}}{S^{5}}\hookrightarrow A(5.9)\), 
where $\alpha_{2}$ generates $\pi_{12}(S^{5})\cong\mathbb{Z}/3\mathbb{Z}$ 
and $t\in\mathbb{Z}/3\mathbb{Z}$. Each map in this composite is stable, 
so $g$ is stable. But by Lemma \ref{lem:top-split}, the top cell of $\Tel(c_{4})$ splits off 
stably. Thus $g$ is stably trivial, and so $g$ itself must be trivial. 
Hence, we have $\Tel(c_{4})\simeq A(5,9)\vee S^{13}$. 
To sum up, there is a $3$-local homotopy equivalence 
\[\Sigma SU(4)/T^{3}\simeq_3 3A(3,7)\vee 3S^{5}\vee 3A(7,11)\vee 3S^{9}\vee 
      2 A(5,9)\vee S^{13}.\] 

If $p=5$ then for degree reasons the unstable Adams operations decompose 
$\Sigma SU(4)/T^{3}$ as a wedge $X_{1}\vee\cdots\vee X_{4}$ where 
$H^{\ast}(X;\Z/5\Z)$ consists of the degree $3$ and $11$ elements in 
$H^{\ast}(SU(4)/T^{3};\Z/5\Z)$, $H^{\ast}(X_{2};\Z/5\Z)$ consists of the 
degree $5$ and $13$ elements, $H^{\ast}(X_{3};\Z/5\Z)$ consists of the degree $7$ 
elements, and $H^{\ast}(X_{4};\Z/5\Z)$ consists of the degree $9$ elements. 
On the other hand, since $\sigma_w^5=0$ for all $l(w)=1$, all $\mathcal{P}^1$ are trivial. 
Thus there is a $5$-local homotopy equivalence 
\[\Sigma SU(4)/T^{3}\simeq_5 3S^{3}\vee 5S^{5}\vee 6S^{7}\vee 5S^{9}\vee 3S^{11}\vee S^{13}.\] 

If $p>5$ then for degree reasons the unstable Adams operations decompose 
$SU(4)/T^{3}$ as a wedge of spheres (the same wedge as in the $p=5$ case.)  
\end{Ex}

\subsection{Type $B_2=C_2$}
The Weyl group $W$ is the hyper-octahedral group $\langle s_1,s_2 \rangle$ with $|W|=8$.
We have $H^*(Sp(2)/T)=\dfrac{\Z[x_1,x_2]}{(x_1^2+x_2^2,x_1^2x_2^2)}$ and 
its Betti numbers are $\{ 1,2,2,2,1 \}$.

For $p=2$, there is no non-trivial idempotents.

For $p>2$ (the non-modular case), there is a set of mutually orthogonal idempotents in $\Z/p\Z[W]$ 
\begin{align*}
c_1 &= \frac18(1+s_{1}-s_{2}-s_{12}-s_{21}-s_{121}+s_{212}+s_{1212})\\
c_2 &= \frac18(1-s_{1}+s_{2}-s_{12}-s_{21}+s_{121}-s_{212}+s_{1212})\\
c_3 &= \frac18(1-s_{1}-s_{2}+s_{12}+s_{21}-s_{121}-s_{212}+s_{1212})\\
c_4 &= \frac18(1-s_{1}+s_{212}-s_{1212})\\
c_5 &= \frac18(1+s_{1}-s_{212}-s_{1212}) \\
c_6 &= \frac18\sum_{w\in W} w  
\end{align*} 
satisfying 
\begin{align*} 
V_1 &= \langle \sigma_{12} \rangle \\
V_2 &= \langle \sigma_{21} \rangle \\
V_3 &= \langle \sigma_{1212} \rangle \\
V_4 &= \langle \sigma_{1}-\sigma_2, \sigma_{121}+\sigma_{212} \rangle \\
V_5 &= \langle \sigma_2, \sigma_{212} \rangle \\
V_6 &= \langle 1 \rangle
\end{align*}
where $\sigma_1^3=\sigma_{121}$ and $\sigma_2^3=2\sigma_{212}$. 
In particular, we obtain $\Tel(c_{1})\simeq\Tel(c_{2})\simeq S^{5}$,  
$\Tel(c_{3})\simeq S^{9}$ and $\Tel(c_{6})\simeq\ast$. 

If $p=3$ then the equations $\sigma_1^3=\sigma_{121}$ and $\sigma_2^3=2\sigma_{212}$ 
imply that $\mathcal{P}^1$ is non-trivial on $\sigma_1$ and~$\sigma_2$. Therefore 
$\Tel(c_{4})\simeq\Tel(c_{5})\simeq A(3,7)$. Hence there is a $3$-local 
homotopy equivalence 
\[\Sigma Sp(2)/T\simeq_3 2A(3,7)\vee 2S^{5}\vee S^{9}.\] 

If $p\geq 5$ then all Steenrod operations $\mathcal{P}^{1}$ are trivial so 
$\Tel(c_{4})\simeq\Tel(c_{5})\simeq S^{3}\vee S^{7}$. Hence there is a 
$p$-local homotopy equivalence 
\[\Sigma Sp(2)/T\simeq_p 2S^{3}\vee 2S^{5}\vee 2S^{7}\vee S^{9}.\]

\subsection{Type $G_2$} 
\label{G2section} 
The Weyl group $W$ is the dihedral group $D_6=\langle s_1,s_2 \rangle$ with $|W|=12$. 
The cohomology of $G_2/T$ is computed by \cite{Bott-Samelson,Toda-Watanabe} as
\[
 H^\ast(G_2/T)=\frac{\Z[x_1,x_2,x_3,\gamma]}{(e_1,e_2,e_3-2\gamma,\gamma^2)}.
\] 
For $p=2$, there is a set of mutually orthogonal idempotents in $\Z/2\Z[W]$
\begin{align*}
c_{1} &= 1+s_{1212}+s_{2121}\\ 
c_{2} &= 1+s_{212}+s_{1212}+s_{12121}\\
c_{3} &= 1+s_{212}+s_{2121}+s_{12121} 
\end{align*} 
satisfying 
\begin{align*} 
V_1 &= \langle 1, \sigma_{121},\sigma_{212},\sigma_{121212} \rangle, \\
V_2 &= \langle \sigma_{1}+\sigma_{2}, \sigma_{12}+\sigma_{21},\sigma_{1212}+\sigma_{2121},\sigma_{12121}+\sigma_{21212} \rangle, \\
V_3 &= \langle \sigma_{1}, \sigma_{21},\sigma_{2121},\sigma_{12121}\rangle,
\end{align*}
where $\Sq^2(\sigma_1)=\sigma_{21},\Sq^2(\sigma_2)=\sigma_{12},\Sq^2(\sigma_{1212})=\sigma_{21212},
\Sq^2(\sigma_{2121})=\sigma_{12121}$. The multiple generators of different degrees in the 
modules $V_{i}$ imply that the telescopes of the maps $c_{i}$ are not readily identifiable. 
So we say nothing more than there is a $2$-local homotopy equivalence 
\[\Sigma G_{2}/T\simeq_2\Tel(c_{1})\vee\Tel(c_{2})\vee\Tel(c_{3}).\] 

For $p=3$, there is a set of mutually orthogonal idempotents in $\Z/3\Z[W]$
\begin{align*}
c_1 &= 1+s_1+s_{21212}+ s_{121212} \\
c_2 &= 1+s_1-s_{21212}- s_{121212} \\
c_3 &= 1-s_1-s_{21212}+ s_{121212} \\
c_4 &= 1-s_1+s_{21212}- s_{121212} 
\end{align*} 
satisfying 
\begin{align*} 
V_1 &= \langle 1, \sigma_{12}, \sigma_{1212} \rangle \\
V_2 &= \langle \sigma_2, \sigma_{212}, \sigma_{21212} \rangle \\
V_3 &= \langle \sigma_{21}+\sigma_{12}, \sigma_{2121}-\sigma_{1212}, \sigma_{121212} \rangle \\
V_4 &= \langle \sigma_{1}+\sigma_{2}, \sigma_{121}, \sigma_{12121}-\sigma_{21212} \rangle
\end{align*}
The non-trivial actions of $\mathcal{P}^1$ are
\begin{align*}
V_1, V_3: &  H^4\to H^{8} \\
V_2: &  H^6\to H^{10} \\
V_4: & H^2 \to H^6.
\end{align*} 
In particular, $\Tel(c_{1})\simeq A(5,9)$, but the other telescopes are not as 
readily identifiable. So we say nothing more right now other than there is a $3$-local 
homotopy equivalence 
\[\Sigma G_{2}/T\simeq_3 A(5,9)\vee\Tel(c_{2})\vee\Tel(c_{3})\vee\Tel(c_{4}).\]

For $p>3$ (the non-modular case), the regular representation decomposes into  
four $1$-dimensional and four $2$-dimensional irreducible representations.
The maps inducing the cohomology idempotents corresponding to the $1$-dimensional 
irreducible representations are given by 
\begin{align*}
c_1 &= \dfrac{1}{12} \sum_{w\in W} w
& c_2 = \dfrac{1}{12} \sum_{w\in W} (-1)^{l(w)} w \\
c_3 &= \dfrac{1}{12} \sum_{w\in W} (-1)^{\#_w 1} w
& c_4 = \dfrac{1}{12} \sum_{w\in W} (-1)^{\#_w 2} w 
\end{align*} 
where $\#_w i$ is the number of $s_i$'s in $w$, and these satisfy 
\begin{align*} 
V_1 &= \langle 1 \rangle
& V_2 = \langle \sigma_{121212} \rangle \\
V_3 &= \langle \sigma_{121} \rangle
& V_4 = \langle \sigma_{212} \rangle. \\
\end{align*} 
In particular, we obtain $\Tel(c_{1})\simeq\ast$, $\Tel(c_{2})\simeq S^{13}$  
and $\Tel(c_{3})\simeq\Tel(c_{4})\simeq S^{7}$. 

The maps inducing the cohomology idempotents corresponding to the $2$-dimensional 
irreducible representations are given by 
\begin{align*}
c_5 &= \dfrac{1}{12}(
2+s_1-2s_2-s_{12}-s_{21}+s_{121}+s_{212}-s_{1212}-s_{2121}-2s_{12121}+
s_{21212}+2s_{121212}) \\
c_6 &= \dfrac{1}{12}(
2-s_1+2s_2-s_{12}-s_{21}-s_{121}-s_{212}-s_{1212}-s_{2121}+2s_{12121}-
s_{21212}+2s_{121212}) \\
c_7 &= \dfrac{1}{6}(1-s_1-s_2+s_{21}-s_{1212}+s_{12121}+s_{21212}-s_{121212}) \\
c_8 &= \dfrac{1}{6}(1+s_1+s_2+s_{12}-s_{2121}-s_{12121}-s_{21212}-s_{121212})  
\end{align*} 
and these satisfy  
\begin{align*} 
V_5 &= \langle 2\sigma_{12}-\sigma_{21}, 2\sigma_{1212}+\sigma_{2121} \rangle \\
V_6 &= \langle \sigma_{21}, \sigma_{2121} \rangle \\
V_7 &= \langle 3\sigma_{1}-2\sigma_2, 3\sigma_{12121}+2\sigma_{21212} \rangle \\
V_8 &= \langle \sigma_{2}, \sigma_{21212} \rangle. \\
\end{align*} 
If $p=5$ then the Steenrod operation $\mathcal{P}^{1}$ is trivial on $V_{5}$ 
and $V_{6}$ for degree reasons. So 
$\Tel(c_{5})\simeq\Tel(c_{6})\simeq S^{5}\vee S^{9}$. 
On the other hand, $\mathcal{P}^1(\sigma_2)=3\sigma_{21212}$ 
and $\mathcal{P}^1(3\sigma_{1}-2\sigma_2)=2(3\sigma_{12121}+2\sigma_{21212})$ 
so $\Tel(c_{7})\simeq\Tel(c_{8})\simeq A(3,11)$. 
Therefore there is a $5$-local homotopy equivalence 
\[\Sigma G_{2}/T\simeq_5 2S^{5}\vee 2S^{7}\vee 2S^{9}\vee S^{13}\vee 2 A(3,11).\] 
If $p>5$ then the Steenrod operation $\mathcal{P}^{1}$ is trivial on each of $V_{5}$, 
$V_{6}$, $V_{7}$ and $V_{8}$. Therefore there is a $p$-local homotopy equivalence 
\[\Sigma G_{2}/T\simeq_p 2S^{3}\vee 2S^{5}\vee 2S^{7}\vee 2S^{9}\vee 2S^{11}\vee S^{13}.\] 

\begin{Rem}
It is interesting to note that factors are ``Poincar\'e dual'' to each other.
For example, $(V_6,V_7)$ in $SU(4)/T$ with $p=3$ are dual to each other 
in the sense that the generators in the complimentary degrees multiply to the top degree element
(e.g., $(\sigma_{2321}+\sigma_{1232}+2\sigma_{2132}+2\sigma_{1213})(\sigma_{21}+\sigma_{23})=\sigma_{121321} \mod 3$).
For $SU(4)/T$ with $p=3$, $(V_1,V_2), (V_3,V_8)$, and $(V_4,V_5)$ are all dual pairs as well.
For $Sp(2)/T$ with $p>2$, $(V_1,V_2), (V_3,V_6)$, and $(V_4,V_5)$ are dual pairs.
For $G_2/T$ with $p>3$, $(V_1,V_2), (V_3,V_4), (V_5,V_6)$, and $(V_7,V_8)$ are dual pairs.
It might be interesting to find a geometric explanation for this.
\end{Rem}

\section{Self-maps of flag manifolds} 
\label{sec:selfmap} 

The decomposition of $\Sigma G/T$ allows for a calculation of the 
factor $[G/T,G]$ of $[G/T,G/T]$. Since $G$ is a topological 
group, it has a classifying space $BG$, and $G\simeq\Omega BG$. Therefore 
there is an adjunction giving an isomorphism of groups 
$[G/T,G]\cong [G/T,\Omega BG]\cong [\Sigma G/T,BG]$. 
Suppose that there is a homotopy decomposition 
$\Sigma G/T\simeq\bigvee_{i=1}^{k} A_{i}$. As this is a decomposition of spaces 
rather than co-$H$-spaces, there is a set isomorphism 
$[\Sigma G/T,BG]\cong [\bigvee_{i=1}^{k} A_{i}, BG]\cong\prod_{i=1}^{k} [A_{i},BG]$. 
Combining these isomorphisms gives the following. 

\begin{Lem} 
   \label{hset} 
   If $\Sigma G/T\simeq\bigvee_{i=1}^{k} A_{i}$ then there is an isomorphism 
   of sets $[G/T,G]\cong\prod_{i=1}^{k} [A_{i},BG]$.~$\qqed$ 
\end{Lem} 

Another useful general lemma is the following. 

\begin{Lem} 
   \label{reduction} 
   If $G$ is simply-connected then, rationally, $[G/T,G]\cong 0$. Consequently, 
   $[G/T,G]$ is the product of its $p$-components for all primes $p$.  
\end{Lem} 

\begin{proof} 
In all cases, there is a rational homotopy equivalence 
$G\simeq\prod_{i=1}^{m} K(\mathbb{Q},2n_{i}-1)$ 
for some sequence $\{i_{1},\ldots,i_{m}\}$. Thus 
$[G/T,G]\cong\prod_{i=1}^{m} H^{2n_{i}-1}(G/T;\mathbb{Q})$. 
But $H^{\mbox{odd}}(G/T;\mathbb{Q})\cong 0$, so we obtain a rational isomorphism 
$[G/T,G]\cong 0$. 
\end{proof} 

We now explicitly calculate $[G/T,G]$ when $G$ is one of $SU(3)$, $SU(4)$, 
$Sp(2)$ or $G_{2}$. 

\begin{Prop} 
   \label{su3T} 
   There is a group isomorphism $[SU(3)/T^{2},SU(3)]\cong\mathbb{Z}/6\mathbb{Z}$ 
   and the generator corresponds to the self-map 
   \[\namedddright{SU(3)/T^{2}}{q}{S^{6}}{f}{SU(3)}{}{SU(3)/T}\] 
   where $q$ is the pinch to the top cell and $f$ represents the generator 
   of $\pi_{6}(SU(3))\cong\mathbb{Z}/6\mathbb{Z}$. 
\end{Prop} 

\begin{proof} 
By Lemma~\ref{reduction}, to calculate $[SU(3)/T^{2},SU(3)]$ it suffices to 
localize and work prime by prime.~\medskip 

\noindent 
\textit{Case 1: $p=2$}. 
By Example~\ref{SU(3)-weyl}, there is a $2$-local homotopy equivalence 
\[\Sigma SU(3)/T^{2}\simeq S^{7}\vee\Sigma\cptwo\vee\Sigma\cptwo.\] 
Therefore 
\begin{align*} 
   [SU(3)/T^{2},SU(3)] & \cong [\Sigma SU(3)/T^{2},BSU(3)] \\ 
   & \cong [S^{7},BSU(3)]\times [\Sigma\cptwo,BSU(3)]\times [\Sigma\cptwo,BSU(3)]. 
\end{align*} 
By~\cite{M}, the $2$-component of $\pi_{6}(SU(3))$ is $\mathbb{Z}/2\mathbb{Z}$. 
Since the homotopy fibre of 
\(\namedright{BSU(3)}{}{BSU(\infty)}\) 
is $6$-connected and $\Sigma\cptwo$ is $5$-dimensional we have 
\[[\Sigma\cptwo,BSU(3)]\cong [\Sigma\cptwo,BSU(\infty)]\cong\widetilde{K}(\Sigma\cptwo)\cong 0\] 
where $\widetilde{K}$ is reduced complex $K$-theory. Therefore 
\[[SU(3)/T^{2},SU(3)]\cong\mathbb{Z}/2\mathbb{Z}\] 
and this corresponds to a nontrivial $2$-local self-map of $SU(3)/T^{2}$ given by the composite 
\[\namedddright{SU(3)/T^{2}}{q}{S^{6}}{f}{SU(3)}{}{SU(3)/T^{2}}\] 
where $q$ is the pinch to the top cell and $f$ represents the generator 
of $\pi_{6}(SU(3))\cong\mathbb{Z}/2\mathbb{Z}$. 
\medskip 

\noindent 
\textit{Case 2: $p>2$}. 
By Example~\ref{SU(3)-weyl}, localized at a prime $p>2$ there is a homotopy equivalence 
$\Sigma SU(3)/T^{2}\simeq 2S^{3}\vee 2S^{5}\vee S^{7}$. 
Therefore, by Lemma~\ref{hset}, 
\begin{align*} 
      [SU(3)/T^{2},SU(3)] & \cong 2[S^{3},BSU(3)]\times  
               2[S^{5},BSU(3)]\times [S^{7},BSU(3)] \\ 
         & \cong 2\pi_{2}(SU(3))\times 2\pi_{4}(SU(3))\times\pi_{6}(SU(3)). 
\end{align*} 
By~\cite{MT}, $\pi_{2}(SU(3))\cong 0$, $\pi_{4}(SU(3))\cong 0$ and after inverting $2$, 
$\pi_{6}(SU(3))\cong\mathbb{Z}/3\mathbb{Z}$. Therefore, localized at $p>3$ we have  
$[SU(3)/T^{2},SU(3)]\cong 0$ and localized at $3$ we have 
\[[SU(3)/T^{2},SU(3)]\cong\mathbb{Z}/3\mathbb{Z}.\] 
In the latter case, the generator corresponds to a nontrivial $3$-local self-map of $SU(3)/T^{2}$ 
given by the composite 
\[\namedddright{SU(3)/T^{2}}{q}{S^{6}}{f}{SU(3)}{}{SU(3)/T^{2}}\] 
where $q$ is the pinch map to the top cell and $f$ represents a generator 
of $\pi_{6}(SU(3))\cong\mathbb{Z}/3\mathbb{Z}$. 
\medskip 

Combining both cases, we obtain a set isomorphism 
$[SU(3)/T^{2},SU(3)]\cong\mathbb{Z}/6\mathbb{Z}$. 
To upgrade this to an isomorphism of groups, it suffices to show that the group 
$[SU(3)/T^{2},SU(3)]$ has an element of order~$6$. But observe that the generators of the $2$ 
and $3$-components are obtained from the same map  
\[\namedddright{SU(3)/T^{2}}{q}{S^{6}}{f}{SU(3)}{}{SU(3)/T^{2}}\] 
where $f$ represents a generator of $\pi_{6}(S^{3})\cong\mathbb{Z}/6\mathbb{Z}$.  
Thus $f\circ q$ has order~$6$ in $[SU(3)/T^{2},SU(3)]$ and we are done. 
\end{proof} 

\begin{Prop} 
   \label{su4T} 
   Localized away from $2$ there is a set isomorphism 
   $[SU(4)/T^{3},SU(4)]\cong\mathbb{Z}/15\mathbb{Z}\times 3(\mathbb{Z}/5\mathbb{Z})$. 
   The group $[SU(4)/T^{3},SU(4)]$ has a subgroup of order $15$ corresponding to the self-map 
   \[\namedddright{SU(4)/T^{3}}{q}{S^{12}}{f}{SU(4)}{}{SU(4)/T^{3}}\] 
   where $q$ is the pinch map to the top cell and $f$ represents the $3$ and $5$-components 
   of $\pi_{12}(SU(4))\cong\mathbb{Z}/60\mathbb{Z}$. The group $[SU(4)/T^{3},SU(4)]$ 
   has three subgroups of order $5$ corresponding to $5$-local self-maps 
   \[\namedddright{SU(4)/T^{3}}{q'}{3S^{10}\vee S^{12}}{p_{i}}{S^{10}}{f_{i}}{SU(4)} 
          \longrightarrow SU(4)/T^{3}\] 
   where $q'$ collapses the $9$-skeleton of $SU(4)/T^{3}$ to a point, and for $1\leq i\leq 3$ 
   the map $p_{i}$ pinches to the $i^{th}$-copy of $S^{10}$ while $f_{i}$ represents the 
   generator of $\pi_{10}(SU(4))\cong\mathbb{Z}/5\mathbb{Z}$. 
\end{Prop} 

\begin{proof} 
By Lemma~\ref{reduction}, to calculate $[SU(4)/T^{3},SU(4)]$ it suffices to 
localize and work prime by prime. 
\medskip 

\noindent 
\textit{Case 1: $p=3$}. By Example~\ref{SU(4)-weyl}, there is a $3$-local homotopy equivalence 
\begin{equation} 
\label{su4at3} 
\Sigma SU(4)/T^{3}\simeq 3A(3,7)\vee 3S^{5}\vee 3A(7,11)\vee 3S^{9}\vee 
      2A(5,9)\vee S^{13}. 
\end{equation} 
We calculate $[\Sigma SU(4)/T^{3},BSU(4)]$ by using Lemma~\ref{hset}. 

By~\cite{MT}, $\pi_{m}(BSU(4))\cong 0$ for 
$m\in\{3,5,7,9\}$, so $[A(3,7),BSU(4)]\cong [S^{5},BSU(4)]\cong [S^{9},BSU(4)]\cong0$, 
and $\pi_{13}(BSU(4))\cong\mathbb{Z}/60\mathbb{Z}$. It remains to consider 
$[A(7,11),BSU(4)]$ and $[A(5,9),BSU(4)]$. 

For $A(7,11)$, the cofibration sequence 
\(\namedddright{S^{7}}{}{A(7,11)}{}{S^{11}}{\alpha_{1}}{S^{8}}\) 
induces an exact sequence 
\[\namedddright{[S^{8},BSU(4)]}{(\alpha_{1})^{\ast}}{[S^{11},BSU(4)]}{}{[A(7,11),BSU(4)]} 
     {}{[S^{7},BSU(4)]}.\] 
On the one hand, $[S^{7},BSU(4)]=\pi_{6}(SU(4))\cong 0$. On the other hand, 
by~\cite{MT}, $[S^{8},BSU(4)]\cong\mathbb{Z}$, $[S^{11},BSU(4)]\cong\mathbb{Z}/3\mathbb{Z}$, 
and $(\alpha_{1})^{\ast}$ is an epimorphism. Thus $[A(7,11),BSU(4)]\cong 0$. 

For $A(5,9)$, the cofibration sequence        
\(\namedddright{S^{5}}{}{A(5,9)}{}{S^{9}}{\alpha_{1}}{S^{6}}\) 
induces an exact sequence 
\[\namedddright{[S^{6},BSU(4)]}{(\alpha_{1})^{\ast}}{[S^{9},BSU(4)]}{}{[A(5,9),BSU(4)]} 
     {}{[S^{5},BSU(4)]}.\]  
On the one hand, $[S^{5},BSU(4)]=\pi_{4}(SU(4))\cong 0$. On the other hand, 
by~\cite{MT}, $[S^{6},BSU(4)]\cong\mathbb{Z}$, $[S^{9},BSU(4)]\cong\mathbb{Z}/3\mathbb{Z}$, 
and $(\alpha_{1})^{\ast}$ is an epimorphism. Thus $[A(5,9),BSU(4)]\cong 0$. 

Therefore, from~(\ref{su4at3}) we obtain a set isomorphism 
\[[SU(4)/T^{3},SU(4)]\cong\mathbb{Z}/3\mathbb{Z}.\] 
This is in fact a group isomorphism. It suffices to find an element of 
$[SU(4)/T^{3},SU(4)]$ whose order when localized at $3$ is $3$. But this is 
given by the map $q\circ f$ in the composite

\[\namedddright{SU(4)/T^{3}}{q}{S^{12}}{f}{SU(4)}{}{SU(4)/T^{3}}\] 
where $q$ is the pinch map to the top cell and $f$ represents the $3$-component 
of $\pi_{12}(SU(4))\cong\mathbb{Z}/60\mathbb{Z}$. The map $f\circ q$ 
has order $3$ in $[SU(4)/T^{3},SU(4)]$ so 
\medskip 

\noindent 
\textit{Case 2: $p\geq 5$}. 
By Example~\ref{SU(4)-weyl}, localized at a prime $p\geq 5$ there is a homotopy equivalence 
\[\Sigma SU(4)/T^{3}\simeq 3S^{3}\vee 5S^{5}\vee 6S^{7}\vee 5S^{9}\vee 3S^{11}\vee S^{13}.\]  
Therefore, by Lemma~\ref{hset}, 
\begin{align*} 
      [SU(4)/T^{3},SU(4)] & \cong 
          3[S^{3},BSU(4)]\times 5[S^{5},BSU(4)]\times 6[S^{7},BSU(4)]\times \\ 
          &  \hspace{1cm} 5[S^{9},BSU(4)]\times 3[S^{11},BSU(4)]\times [S^{13},BSU(4)].  
\end{align*} 
By~\cite{MT}, $\pi_{2}(SU(4))\cong 0$, $\pi_{4}(SU(4))\cong 0$, $\pi_{6}(SU(4))\cong 0$, 
and after inverting $2$ and $3$, $\pi_{8}(SU(4))\cong 0$, $\pi_{10}(SU(4))\cong\mathbb{Z}/5\mathbb{Z}$ 
and $\pi_{12}(SU(4))\cong\mathbb{Z}/5\mathbb{Z}$. Thus 
\[[SU(4)/T^{3},SU(4)]\cong 4(\mathbb{Z}/5\mathbb{Z})\] 
and the generators correspond to two types of nontrivial $5$-local self-maps. First,  
\[\namedddright{SU(4)/T^{3}}{q'}{3S^{10}\vee S^{12}}{p_{i}}{S^{10}}{f_{i}}{SU(4)} 
         \longrightarrow SU(4)/T^{3}\] 
where $q'$ is the map that collapses out the $9$-skeleton of $SU(4)/T^{3}$, and for $1\leq i\leq 3$ 
the map $p_{i}$ pinches to the $i^{th}$-copy of $S^{10}$ while $f_{i}$ represents the 
generator of $\pi_{10}(SU(4))\cong\mathbb{Z}/5\mathbb{Z}$. Second, 
\[\namedddright{SU(4)/T^{3}}{q}{S^{12}}{f}{SU(4)}{}{SU(4)/T^{3}}\] 
where $q$ is the pinch map to the top cell and $f$ represents the generator 
of $\pi_{12}(SU(4))\cong\mathbb{Z}/5\mathbb{Z}$. 
\medskip 

Finally, notice that the same map 
\(\nameddright{SU(4)/T^{3}}{q}{S^{12}}{f}{SU(4)}\) 
appears in the $p=3$ and $p=5$ cases, so $f\circ q$ has order~$15$ and 
generates a subgroup of order~$15$ in $[SU(4)/T^{3},SU(4)]$. 
\end{proof} 
   
\begin{Prop} 
   \label{g2T} 
   There is a group isomorphism $[G_{2}/T,G_{2}]\cong 0$. 
\end{Prop} 

\begin{proof} 
By Lemma~\ref{reduction}, to calculate $[G_{2}/T,G_{2}]$ it suffices to 
localize and work prime by prime. 
\medskip  

\noindent
\textit{Case 1: $p=2$}. As in Section~\ref{G2section}, there is a $2$-local homotopy 
equivalence 
\[\Sigma G_{2}/T\simeq\Tel(c_{1})\vee\Tel(c_{2})\vee\Tel(c_{3}).\] 
By Lemma~\ref{hset}, to calculate the $2$-component of $[G_{2}/T,G_{2}]$ 
it is equivalent to calculate $[\Tel(c_{i}),BG_{2}]$ for $1\leq i\leq 3$. The space 
$\Tel(c_{1})$ has cells in dimensions $7$ and $13$. By~\cite{M}, 
$\pi_{7}(BG_{2})\cong\pi_{13}(BG_{2})\cong 0$, so $[\Tel(c_{1}),BG_{2}]\cong 0$. 
The spaces $\Tel(c_{2})$ and $\Tel(c_{3})$ both have cells in 
dimensions $3,5,9,11$, and the Steenrod operation $Sq^{2}$ connects the $3$ and $5$ cells, 
and the $9$ and $11$ cells. Therefore, for $2\leq i\leq 3$ there is a homotopy cofibration 
\(\nameddright{\Sigma\cptwo}{}{\Tel(c_{i})}{}{\Sigma^{7}\cptwo}\).  
Let 
\(g\colon\namedright{\Tel(c_{i})}{}{BG_{2}}\) 
be any map. By~\cite{M}, $\pi_{3}(BG_{2})\cong\pi_{5}(BG_{2})\cong 0$, so 
the restriction of $g$ to $\Sigma\cptwo$ is null homotopic, implying that $g$ 
factors as a composite 
\(\nameddright{\Tel(c_{i})}{}{\Sigma^{7}\cptwo}{h}{BG_{2}}\) 
for some map $h$. By~\cite{M}, $\pi_{9}(BG_{2})\cong\mathbb{Z}/2\mathbb{Z}$. 
We claim that the restriction of $h$ to $S^{9}$ is trivial. If not, then 
it represents the generator $\gamma$ of $\pi_{9}(BG_{2})$. By~\cite{M}, this generator 
has the property that the composite 
\(\nameddright{S^{10}}{\eta}{S^{9}}{\gamma}{BG_{2}}\) 
is also nontrivial. But this implies that there can be no extension of $\gamma$ 
to a map 
\(\namedright{\Sigma^{7}\cptwo}{}{BG_{2}}\). 
That is, the restriction of $h$ to~$S^{9}$ cannot extend to $h$, a contradiction. 
Therefore the restriction of $h$ to~$S^{9}$ is trivial, implying that $h$ factors 
as a composite 
\(\nameddright{\Sigma^{7}\cptwo}{}{S^{11}}{k}{BG_{2}}\) 
for some map $k$. By~\cite{M}, $\pi_{11}(BG_{2})\cong 0$. Hence $k$, and 
therefore $h$, and therefore $g$ are all trivial. Consequently, $[\Tel(c_{i}),BG_{2}]\cong 0$ 
for $2\leq i\leq 3$. Collectively, we obtain a $2$-local isomorphism $[G_{2}/T,G]\cong 0$.  
\medskip 

\noindent 
\textit{Case 2: $p=3$}. 
As in Section~\ref{G2section}, there is a $3$-local homotopy equivalence 
\[\Sigma G_{2}/T\simeq A(5,9)\vee\Tel(c_{2})\vee\Tel(c_{3})\vee\Tel(c_{4}).\] 
By Lemma~\ref{hset}, to calculate the $3$-component of $[G_{2}/T,G_{2}]$ it is equivalent 
to calculate $[A(5,9),BG_{2}]$ and $[\Tel(c_{i}),BG_{2}]$ for $2\leq i\leq 4$.
By~\cite{M}, at $3$ we have $\pi_{m}(BG_{2})=0$ for $m\in\{3,5,9,11,13\}$ 
while $\pi_{7}(BG_{2})\cong\mathbb{Z}/{3}\mathbb{Z}$. In particular, 
$[A(5,9),BG_{2}]\cong 0$ since $A(5,9)$ has cells in dimensions $5$ and $9$, 
and $[\Tel(c_{3}),BG_{2}]\cong 0$ as $\Tel(c_{3})$ has cells in 
dimensions $5$, $9$ and $13$. 

The space $\Tel(c_{2})$ has cells in dimensions $3$, $7$ and $11$ with the $7$ 
and $11$ cells connected by the Steenrod operation $\mathcal{P}^{1}$. The 
triviality of $\pi_{3}(BG_{2})$ implies that any map 
\(\namedright{\Tel(c_{3})}{}{BG_{2}}\) 
factors through $\Tel(c_{3})/S^{3}$. The nontrivial Steenrod operation in cohomology 
implies that there is a homotopy cofibration 
\[\namedddright{S^{10}}{\alpha}{S^{7}}{}{\Tel(c_{3})/S^{3}}{}{S^{11}}.\] 
This induces an exact sequence 
\[\namedddright{[S^{11},BG_{2}]}{}{[\Tel(c_{3})/S^{3},BG_{2}]}{}{[S^{7},BG_{2}]}{\alpha^{\ast}} 
       {[S^{10},BG_{2}]}.\] 
On the one hand, $[S^{11},BG_{2}]=\pi_{11}(BG_{2})\cong 0$. On the other hand, 
by~\cite{M}, $[S^{7},BG_{2}]\cong\mathbb{Z}/3\mathbb{Z}$, 
$[S^{10},BG_{2}]\cong\mathbb{Z}/3\mathbb{Z}$, and $\alpha^{\ast}$ is an isomorphism. 
Thus $[\Tel(c_{3})/S^{3},BG_{2}]\cong 0$ and hence $[\Tel(c_{3}),BG_{2}]\cong 0$.  

The space $\Tel(c_{4})$ also has cells in dimensions $3$, $7$ and $11$, but this 
time there is no Steenrod operation connecting the $7$ and $11$ cells. Thus 
$\Tel(c_{4})/S^{3}\simeq S^{7}\vee S^{11}$. As in the $\Tel(c_{2})$ case, any map 
\(\namedright{\Tel(c_{4})}{}{BG_{2}}\) 
factors through $\Tel(c_{4})/S^{3}\simeq S^{7}\vee S^{11}$. The homotopy cofibration 
\[\namedddright{S^{3}}{}{\Tel(c_{4})}{}{S^{7}\vee S^{11}}{\gamma}{S^{4}}\] 
induces an exact sequence 
\begin{equation} 
  \label{X4exact} 
  \namedddright{[S^{4},BG_{2}]}{\gamma^{\ast}}{[\Tel(c_{4}),BG_{2}]}{} 
       {[S^{7}\vee S^{11},BG_{2}]}{}{[S^{3},BG_{2}]}. 
\end{equation}  
Since the $3$ and $7$-cells of $\Tel(c_{4})$ are connected by the Steenrod operation 
$\mathcal{P}^{1}$, the restriction of $\gamma$ to $S^{7}$ is $\alpha$. It is not 
clear what the restriction of $\gamma$ to $S^{11}$ is but this is not relevant since 
$\pi_{11}(BG_{2})=0$, so $\gamma^{\ast}$ factors as 
\(\nameddright{[S^{4},BG_{2}]}{\alpha^{\ast}}{[S^{7},BG_{2}]}{}{[S^{7}\vee S^{11},BG_{2}]}\). 
By~\cite{M}, $[S^{4},BG_{2}]\cong\mathbb{Z}$, $[S^{7},BG_{2}]\cong\mathbb{Z}/3\mathbb{Z}$, 
and $\alpha^{\ast}$ is reduction mod-$3$. Therefore $\gamma^{\ast}$ is onto. On the 
other hand, $[S^{3},BG^{2}]\cong 0$, so exactness in~(\ref{X4exact}) implies that 
$[\Tel(c_{4}),BG_{2}]\cong 0$. 
Collectively, we obtain a $3$-local isomorphism $[G_{2}/T,G_{2}]\cong 0$. 
\medskip 

\noindent 
\textit{Case 3: $p=5$}. 
As in Section~\ref{G2section}, there is a $5$-local homotopy equivalence 
\[\Sigma G_{2}/T\simeq 2 S^{5}\vee 2 S^{7}\vee 2 S^{9}\vee S^{13}\vee 2 A(3,11).\] 
Therefore by Lemma~\ref{hset} 
\[[G_{2}/T,G_{2}]\cong 2\pi_{4}(G_{2})\times 2\pi_{6}(G_{2})\times 2\pi_{8}(G_{2}) 
      \times\pi_{12}(G_{2})\times 2[A(3,11),BG_{2}].\] 
By~\cite{M}, $\pi_{4}(G_{2})\cong 0$ and $\pi_{12}(G_{2})\cong 0$, and localized 
at $5$, $\pi_{6}(G_{2})\cong 0$, $\pi_{8}(G_{2})\cong 0$. As well, 
$\pi_{3}(BG_{2})\cong 0$ and $\pi_{11}(BG_{2})\cong 0$ so $[A,BG_{2}]\cong 0$. 
Thus, at $5$, $[G_{2}/T,G_{2}]\cong 0$. 
\medskip 

\noindent 
\textit{Case 4: $p> 5$}. 
As in Section~\ref{G2section}, there is a $p$-local homotopy equivalence 
\[\Sigma G_{2}/T\simeq 2S^{3}\vee 2 S^{5}\vee 2 S^{7}\vee 2 S^{9}\vee 2S^{11}\vee S^{13}.\] 
By~\cite{M}, the $p$-component of $\pi_{m}(G_{2})$ is $0$ for 
$m\in\{2,4,6,8,10,12\}$. Thus, at $p>5$, $[G/T,G]\cong 0$. 
\end{proof}

\begin{Prop} 
   \label{Sp2T} 
   Localized away from $2$ there is a group isomorphism $[Sp(2)/T,Sp(2)]\cong 0$. 
\end{Prop} 

\begin{proof} 
The cells of $Sp(2)/T$ occur in dimensions $2,4,6,8$, and by~\cite{MT}, 
after inverting $2$ we have $\pi_{m}(Sp(2))\cong 0$ for $m\in\{2,4,6,8\}$. 
Thus $[Sp(2)/T,Sp(2)]\cong 0$. 
\end{proof}

 

\begin{thebibliography}{James}
\bibitem{bergeron} F. Bergeron, N. Bergeron, R.B. Howlett, D.E. Taylor,
A decomposition of the descent algebra of a finite Coxeter group, J. Algebraic Combin. 
\textbf{1} (1992) 23--44.    
\bibitem{BGG} I.N. Bernstein, I.M. Gelfand and S.I. Gelfand,
Schubert cells and the cohomology of the spaces $G/P$,
LMS \textbf{69},  Cambridge Univ. Press (1982), 115--140.
\bibitem{Borel1953} A. Borel,
Sur la cohomologie des espaces fibr\'{e}s principaux et des
   espaces homog\`{e}nes de groupes de Lie compacts,
Ann. of Math. \textbf{57} (1953), 115-207.
\bibitem{Borel1961} A. Borel,
Sous-groupes commutatifs et torsion des groupes de Lie compacts connexes,
T\^ohoku Math. J. (2)  \textbf{13}  (1961), 216--240.
\bibitem{Bott-Samelson} R. Bott and H. Samelson,
The integral cohomology ring of $G/T$, 
Proc. Nat. Acad. Sci. USA \textbf{41} (1955), 490--493.
\bibitem{DZ} H. Duan and X.A. Zhao, The classification of cohomology endomorphisms 
   of certain flag manifolds, {Pacific J. Math.} \textbf{192} (2000), 93-102.
\bibitem{DZ2007} H. Duan and X.A. Zhao, 
A unified formula for Steenrod operations in flag manifolds,
 Compositio Mathematica, \textbf{143} (2007), 257--270.
\bibitem{Che} C. Chevalley,
Sur les d\'{e}composition cellulaires des espaces $G/B$,
Algebraic Groups and their Generalizations: Classical Methods (W. Haboush, ed.),
Proc. Sympos. Pure Math., \textbf{56},  Part 1, Amer. Math. Soc., 1994, 1--23. 
\bibitem{GH} H. Glover and W. Homer, Self-maps of flag manifolds, 
   {Trans. Amer. Math. Soc.} \textbf{267} (1981), 423-434. 
\bibitem{HMR}
P. Hilton, G. Mislin, and J. Roitberg, 
Localization of nilpotent groups and spaces, 
Math. Studies No. 15, North-Holland, Amsterdam, 1975.
\bibitem{maple} S. Kaji, 
Three presentations of torus equivariant cohomology of flag manifolds, 
to appear in Proceedings of International Mathematics Conference in honour of the 70th Birthday of Professor S. A. Ilori,
arXiv.org/1504.01091.
\bibitem{K} N. Kitchloo,
Cohomology splittings of Stiefel manifolds,
J. London Math. Soc. (2) 64 (2001), no. 2, 457--471.
\bibitem{LS} A. Lascoux and M. Sch\"utzenberger,
Polyn\^omes de Schubert,
C. R. Acad. Sci. Paris S\'er. I Math. \textbf{294} (1982), no. 13, 447--450.
\bibitem{maypont}
J. P. May and K. Ponto, More concise algebraic topology, 
Chicago Lectures in Mathematics. University of Chicago Press, Chicago, IL, 2012.
\bibitem{miller} H. Miller, 
Stable splitting of Stiefel manifolds,
Topology, 24(4) (1985), pp. 411--419. 
\bibitem{M} M. Mimura, The homotopy groups of Lie groups of low rank, 
   {J. Math. Kyoto Univ.} \textbf{6} (1967), 131-176. 
\bibitem{MT} M. Mimura and H. Toda, Homotopy groups of $SU(3)$, 
   $SU(4)$ and $Sp(2)$, {J. Math. Kyoto Univ.} \textbf{3} (1964), 217-250. 
\bibitem{NY} G. Nishida and Y. Yang,
On a p-local stable splitting of $U(n)$, 
J. Math. Kyoto Univ. 41 (2001), no. 2, 387--401. 
\bibitem{P} S. Papadima, Rigidity properties of compact Lie groups modulo maximal 
   tori, {Math. Ann.} \textbf{275} (1986), 637-652. 
\bibitem{Priddy} S. Priddy, Recent progress on stable splittings, Proc. Durham Symp. on 
Homotopy Theory 1985, LMS \textbf{117} (1987), 149--174, Cambridge Univ. Press.  
\bibitem{Stembridge2010} J.R. Stembridge, Orthogonal sets of Young symmetrizers,
   {Adv. in Appl. Math.} \textbf{46} (2011), 576--582. 
\bibitem{Toda} H. Toda, {Composition methods in homotopy groups of spheres},
   Annals of Mathematics Studies \textbf{49}, Princeton University Press, 1962. 
\bibitem{Toda-Watanabe} H. Toda and T. Watanabe, 
The integral cohomology ring of $F_{4}/T$ and $E_{6}/T$, 
J. Math. Kyoto Univ. \textbf{14} (1974), 257--286.
\bibitem{Ullman} H. E. Ullman,
An equivariant generalization of the Miller splitting theorem,
Algebr. Geom. Topol. 12 (2012), no. 2, 643--684. 
\bibitem{Y} Y. Yang, On a p-local stable splitting of Stiefel manifolds,
   {J. Math. Soc. Japan} \textbf{54} (2002), 911-921.   
\bibitem{Z1} X.A. Zhao, Cohomology endomorphisms of flag manifolds, 
   {Acta Math. Sinica} \textbf{44} (2001), 1099-1106.  
\bibitem{Z2} X.A. Zhao, Maps from a simply connected space to a flag manifold $G/T$, 
   {Acta. Math. Sinica} \textbf{20} (2004), 1131-1134. 



\end{thebibliography}
\end{document}